\documentclass[a4paper]
{article}

\usepackage{amssymb,amsmath, txfonts,mathrsfs}
\usepackage{xcolor}
\makeindex

\input xy  \xyoption{all}

\newtheorem{deff}{Definition}[section]
\newtheorem{lemma}[deff]{Lemma}

\newtheorem{theorem}[deff]{Theorem}

\newtheorem{corollary}[deff]{Corollary}

\newtheorem{proposition}[deff]{Proposition}

\newtheorem{em-example}[deff]{Example}
\newtheorem{em-examples}[deff]{Examples}
\newtheorem{em-def}[deff]{Definition}        
\newtheorem{em-remark}[deff]{Remark}         
\newtheorem{em-remarks}[deff]{Remarks}
\newtheorem{em-question}[deff]{Question}

\newtheorem{problem}[deff]{Problem}
\newtheorem{claim}[deff]{Claim}

\newenvironment{example}{\begin{em-example} \em }{ \end{em-example}}
\newenvironment{examples}{\begin{em-examples} \em }{ \end{em-examples}}
\newenvironment{definition}{\begin{em-def} \em  }{ \end{em-def}}
\newenvironment{remark}{\begin{em-remark} \em }{\end{em-remark}}

\newenvironment{question}{\begin{em-question}\em }{\end{em-question}}
\newenvironment{proof}{\noindent {\it Proof}:}{\QED \smallskip}

\newcommand{\ol}[1]{\overline{#1}}

\newcommand{\wt}[1]{\widetilde{#1}}
\newcommand{\dis}{\displaystyle}

\newcommand{\tr}{^{\triangleright}}

\newcommand{\VP}{\mathbf{VP}}
\newcommand\QED{\hfill QED \medskip}

\newcommand{\abs}[1]{\left| #1 \right|}

\def\hull#1{\langle#1\rangle}

\def\sign{\mathop{\rm sign}}
\def\CHom{\mathop{\rm CHom}}
\def\Hom{\mathop{\rm Hom}}

\def\:{\nobreak \hskip .1111em\mathpunct {}\nonscript \mkern
-\thinmuskip {:}\hskip .3333emplus.0555em\relax}

\catcode`\@=12
\def\T{{\mathbb T}}
\def\Q{{\mathbb Q}}

\def\Z{{\mathbb Z}}
\def\N{{\mathbb N}}
\def\R{{\mathbb R}}
\def\Q{{\mathbb Q}}

\def\C{{\mathbb C}}

\def\G{\mathcal{G}}

\def\cont{\mathfrak c}
\def\c{\cont}

\def\taub{\tau_\mathbf{b}}

\def\b{\mathbf{b}}

\def\Zpinfty{\Z({p}^\infty)}

\def\LQC{\mathcal{L}qc}

\def\Lin{\mathcal{L}in}
\def\Lp{\mathcal{L}pc}
\def\B
{\overline{\mathcal{V}}_{abs}}
\def\tauc{\tau_\mathbf{c}}

\title{``Varopoulos paradigm": Mackey property vs.  metrizability in topological groups.} 
\author{L. Au\ss{}enhofer\footnote{e-mail: lydia.aussenhofer@uni-passau.de}, D. de la Barrera Mayoral\footnote{ORCID of the second author: 0000-0002-0024-5265. e-mail: dbarrera@mat.ucm.es }, D. Dikranjan
\footnote{The third named author was supported by the grant ``Progetti di Eccellenza 2011/12" of Fondazione CARIPARO. e-mail: dikran.dikranjan@uniud.it},
E. Mart\' \i n Peinador\footnote{ORCID of the fourth author: 0000-0003-0162-0394. e-mail: peinador@mat.ucm.es\newline The first and the third named authors are partially supported by Ministerio de Econom\'ia y Competitividad grant: MTM2013-42486-P.}}

\begin{document}

\maketitle

\begin{abstract}
The class of all locally quasi-convex (lqc)  abelian groups  contains all locally convex vector spaces (lcs) considered as topological groups. Therefore it is natural to extend classical properties of locally convex spaces to this larger class of abelian topological groups. In the present paper we consider the following well known property of lcs: "A metrizable locally convex space carries its Mackey topology ". This claim cannot be extended to lqc-groups in the natural way, as we have recently proved with other coauthors (\cite{AD}, \cite{DMPT} ). We say that an abelian  group $G$ satisfies the
\emph{ Varopoulos paradigm} (VP) if any metrizable locally quasi-convex topology on $G$ is the Mackey topology. Surprisingly VP - which is a topological property of the group - is characterizes an algebraic feature, namely being of finite exponent.
.

\end{abstract}



\section{Introduction.}
\subsection{Historical background and some facts on the Mackey topology.}

 Locally convex topologies  on linear  spaces are a topic which has boosted a wide research in Functional Analysis. It
 was proved by G. Mackey \cite{mack46} that the family of compatible locally convex topologies on a locally convex space has a top  element. That is, for a (real or complex)  locally convex  space $E$ with dual space $E'$, there exists a locally convex topology $\mu=\mu(E,E')$ with $(E,\mu)'=E'$, such that any other   locally convex topology $\nu$ on $E$
 the  dual of which is $E'$, satisfies the inequality $\nu\leq\mu$. The topology $\mu$ satisfying these conditions is named the {\em Mackey topology} of $E$. The following result will be the inspiring source for our reflections on abelian topological groups.

\begin{theorem}\label{MackeyThm}
\cite[Corollary 22.3, page 210]{LTS}
Metrizable locally convex linear topological spaces carry their Mackey topology.
\end{theorem}

 A parallelism between real topological vector spaces and abelian topological groups  can be established if the dualizing object in the framework of abelian topological groups is taken as the circle group $\T$ (or equivalently the quotient group $\R / \Z$). The role of linear forms is played now by characters, and the dual group of a topological abelian group $G$ is now $G^\wedge: = C Hom (G, \T)$. Another group topology on a topological group $G$ will be called compatible if it gives rise to the same dual group $G^\wedge$. These basic concepts permit to study duality in the class of abelian topological groups.

The notion of quasi-convex subsets of a topological abelian group (Vilenkin, \cite{Vil}) paved the way to
formulate the questions underlying the Mackey theory from topological vector spaces to the class of abelian groups. A decade after the introduction of locally quasi-convex groups, Varopoulos tried to extend the Mackey theory to abelian groups, seemingly unaware of the work of Vilenkin. For this goal he focused on locally precompact groups, a subclass  of the class of locally quasi-convex groups. He proved that a metrizable locally precompact topological abelian group carries the finest compatible locally precompact group topology \cite{V}. That is, a metrizable locally precompact group is Mackey in the class of locally precompact groups (or $\Lp$-Mackey, in terms of Definition \ref{Def-G-Mackey}).

In 1999 in \cite{CMPT} a framework was established to generalize the Mackey theory for locally convex spaces to the larger class of locally quasi-convex groups.
%
Although a topological vector space is locally convex if and only if it is locally quasi-convex as a topological group \cite[2.4]{B}, there are some obstructions to extend the Mackey-Arens theorem to the class of locally quasi-convex groups, so the Mackey problem is stated as follows in \cite{CMPT}:

\begin{problem}{\rm \cite{CMPT}}
Let $(G,\tau)$ be a locally quasi-convex topological group. Let $\mathcal{C}(G,\tau)$ denote the family of all locally quasi-convex group topologies which are compatible with $\tau$.
\begin{enumerate}
\item Does there exist a top element  in the family $\mathcal{C}(G,\tau)$?
\end{enumerate}
If it exists, we denote it by $\mu=\mu(G,\tau)$. Then $(G,\mu)$ will be called a Mackey group and $\mu$ the Mackey topology for $(G,\tau)$.
\end{problem}

In the class $\LQC$ there always exists a bottom element in  $\mathcal{C}(G,\tau)$, namely the weak topology on $G$
induced by $(G,\tau)^\wedge$. So in case  a precompact topology is Mackey (cf. Theorem B (ii)), then it is the
unique compatible topology.\footnote{This statement was contained in the former Corollary B 2.}

This problem is a specialization of a more general one. In order to formulate it we need the following:

\begin{definition}\label{Def-G-Mackey} Let $\G$ be a class of abelian topological groups and let $(G, \tau)$ be a topological group in $\G$. Let $\mathcal{C}_\G(G,\tau)$ denote the family of all $\G$-topologies $\nu$ on $G$ compatible with $\tau$.
We say that $\mu\in \mathcal{C}_\G(G,\tau)$ is the {\em $\G$-Mackey topology} for $(G,G^\wedge)$ (or the {\em Mackey topology for $(G,\tau)$ in $\G$}) if $\nu\leq\mu$ for all $\nu\in \mathcal{C}_\G(G,\tau)$. \\
If $\G$ is the class of locally quasi-convex groups $\LQC$, we will simply say that the topology is Mackey.
\end{definition}

\begin{remark}\label{ExM}
The Mackey topology exists in $\G$ if and only if the supremum of two compatible topologies in $\G$ is again
compatible.\footnote{Do we need  a proof here?\\
Who is asking this ? As far as I remember, this is proved in \cite{CMPT}, maybe not exactly in this form (dd)\\
L: I do not think that we have to give a proof.}
\end{remark}

\begin{problem}
Let $\G$ be a class of topological abelian groups. Does there exist a top element  in the family $\mathcal{C}_\G(G,\tau)$, i.e., does the $\G$-Mackey topology exist ?
\end{problem}

\footnote{L: We could add the following Remark:
According to Barr and Kleisli, the Mackey topology exists in the class of all nuclear groups. The latter class was introduced by Banaszczyk in \cite{B} and contains all locally compact abelian groups and all nuclear vector spaces.}

The problem whether the Mackey topology exists, in the setting of $\LQC$-groups, is so far open. Under some constrains on the starting topological group it does. For instance, in the following classes of topological groups, the Mackey topology exists and it is precisely the original one:\footnote{L I think the last sentence is somehow misleading, since we know almost nothing about the existence of the Mackey topology in LQC; I suggest to write: Although almost nothing is known
about the existence of the Mackey topology in $\LQC$, some particular topological properties are sufficient for a topology to be the Mackey topology.}

\begin{theorem}\cite{CMPT}\label{TheoremCMPT}
Let $(G,\tau)$ be a locally quasi-convex topological group satisfying one of the following conditions:
\begin{enumerate}
 \item $G$ is Baire and separable.
 \item $G$ is \v{C}ech-complete (in particular, if  $G$ is metrizable and complete).
 \item $G$ is pseudocompact.
\end{enumerate}
Then $\tau$ is the Mackey topology for $(G,\tau)$.
\end{theorem}

The fact that every metrizable and complete locally quasi-convex group is Mackey, points out an analogy between topological groups and linear topological spaces where the counterpart of this result is available (see Theorem \ref{MackeyThm}).

There are metrizable locally quasi-convex groups which are not Mackey. Thus completeness seems to be an essential requirement.
\begin{example}\label{exampleUnboundedDadidi}
The following groups admit a metrizable non-discrete locally quasi-convex topology which is not Mackey. Since they are countable they cannot be complete, for otherwise, by  the Baire Category Theorem they would be discrete:
\begin{itemize}
 \item[(a)]  The group of the integers $\Z$ endowed with any non-discrete  linear topology is metrizable and locally quasi-convex but fails to be Mackey (\cite{AD}).
 \item[(b)]  The Pr\"{u}fer groups $\Zpinfty$ endowed with the topology inherited from $\T$ are metrizable locally-quasi-convex groups which are not Mackey (\cite{DlB}).
 \item[(c)]  The group of rationals $\Q$ endowed with the topology inherited from the real line is a {\em non precompact} metrizable locally quasi-convex group which is not Mackey (\cite{DlB,DE}).
\end{itemize}
\end{example}

\begin{example} The following groups are metrizable locally quasi-convex and Mackey:

Let $m\geq 2$ be a natural number and consider the direct sum of countable many copies of $\Z_m$, $G_m:=\bigoplus_\omega\Z_m$ endowed with the topology inherited from the product $\prod_\omega\Z_m$. Then $G_m$ is a metrizable, locally quasi-convex {\em non-complete} Mackey group (\cite{BTM}).
\end{example}

One can also find examples with a stronger flavor from Functional Analysis, namely connected, metrizable topological groups which are not Mackey. Indeed, the group of all the null-sequences on $\T$,  endowed with the topology inherited from the product $\T^\N$ is a metrizable connected precompact locally quasi-convex group topology which is not Mackey (\cite{DMPT}).

\subsection{Main results.}

Metrizability alone does not ensure that a locally quasi-convex topology on a group is Mackey, unless it is accompanied (or replaced) by an additional topological property (completeness, etc., see Theorem \ref{TheoremCMPT}). This motivates to investigate which additional properties of {\em algebraic} nature, that when imposed along with metrizability   guarantees that a locally quasi-convex topology group is Mackey. 

\begin{definition} We say that an abelian group $G$ satisfies the {\em Varopoulos Paradigm}, if every metrizable locally quasi-convex topology on $G$ must be a Mackey topology. We briefly denote this by $G \in \VP$.
\end{definition}

According to the main theorem of this paper, the groups satisfying the Varopoulos Paradigm are exactly the bounded ones:

\vspace{0.2cm}
\noindent\textbf{Theorem A: }{\em Let $G$ be an abelian group. Then, the following assertions are equivalent:

\begin{itemize}
 \item[$(i)$] $G \in \VP$, i.e., every metrizable locally quasi-convex group topology on $G$ is Mackey.
 \item[$(ii)$] $G$ is bounded.
\end{itemize}
}
\vspace{0.2cm}

For the proof of Theorem A, we will consider separately the cases of bounded and unbounded groups. In each case we will get one of the implications of Theorem A. In fact, Theorem B is a stronger version of the implication $(ii)\Rightarrow (i)$  in Theorem A, while Theorem C is simply the implication $(i)\Rightarrow (ii)$.

The main result in the case of bounded groups is:

\vspace{0.2cm}
\noindent\textbf{Theorem B: }{\em Let $G$ be a bounded abelian group and $\tau$ a locally quasi-convex topology on $G$.
\begin{itemize}
\item[$(i)$]  If $\tau$ is metrizable, then it is the Mackey topology for $G$.
\item[$(ii)$] If $\abs{(G,\tau)^\wedge}<\c$ then $\tau$ is precompact and the Mackey topology for $G$.
\end{itemize}
Then $(G,\tau)$ is Mackey.}
\vspace{0.2cm}

\begin{remark}
In \cite[Example 4]{BTM} a weaker version of Theorem B (ii) was shown, namely: Every bounded precompact topological group $(G,\tau)$ with $\abs{(G,\tau)^\wedge}<\c$ is Mackey.

In \cite[Prop.8.57]{LL} Theorem B (ii) was proved by other methods.

\end{remark}

 Immediately  from Theorem B (ii) we get:

\begin{corollary}\label{corollaryLemma2.16}
Let $(G,\tau)$ be a 
locally quasi-convex group topology on a bounded group which is not precompact. Then $\abs{(G,\tau)^\wedge}\geq\c$.\footnote{In the previous version $G$ was assumed to be metrizable, what is not necessary.}
\end{corollary}

\begin{examples}\label{ex} The following examples shall show that none of the hypotheses of Theorem B can be omitted:
\begin{enumerate}
\item[(a)] The $p$-adic topology on $\Z$ is locally quasi--convex (even linear), metrizable, and precompact, but it is not Mackey (\cite{AD}).
Its dual group has cardinality $\c$.

\item[(b)]   Consider a $D$-sequence $\b=(b_n)$ satisfying $\frac{b_{n+1}}{b_n}\rightarrow\infty$ as in \cite{AD}. Let $\taub$ be the topology on $\Z$ of uniform convergence   on the set $ \left\{\frac{1}{b_n}+\Z\right\}\subset \T$. As proved in \cite{AD},  $\taub$ is metrizable and locally quasi-convex and satisfies   $\abs{(\Z,\taub)^\wedge}<\cont$ but it is not precompact.

\item[(c)] The Prüfer group $\Z(p^\infty)$ endowed with the topology from $\T$ is a prcompact, metrizable torsion group with countable dual group which is not Mackey according to \cite{DlB}. This shows that we cannot replace "bounded" by "torsion groups".
\item[(d)] The direct sum of $\c$-many copies of $\Z_m$ endowed with the topology inherited from the product $\Z_m^\R$ is precompact (and linear) and bounded, but its dual group has cardinality $\c$.\footnote{L: Do we know whether this group is Mackey? If so, it would be good to give an argument here.} This shows that the first implication of Theorem B (ii) cannot be inverted.
\item[(e)]    The group $G=\bigoplus_\omega\Z_4$ endowed with the discrete topology is bounded, locally quasi-convex  with $\abs{G^\wedge}=\cont$, but it is not precompact.
\item[(f)] Let $G$ be an infinite (bounded) group. Let $\delta^+$ denote the Bohr topology on $G$, i.e. the weak topology on $G$ induced by all homomorphisms $G\to\T$. Then $(G,\delta^+)$ is  (bounded,) precompact,
    $|(G,\delta^+)|=2^{|G|}\ge \c$ and it is not Mackey, as the discrete topology is Mackey topology for
    $(G,(G,\delta^+)^\wedge)$.
\end{enumerate}
\end{examples}



From Theorem B (and some extra results from \cite{ADM1}), we compute the cardinality of the family $\mathcal{C}(G,\tau)$ for a metrizable bounded group, as expressed next:

\vspace{0.2cm}
\noindent\textbf{Corollary B1: }Let $(G,\tau)$ be a metrizable, locally quasi-convex and bounded group. Then:
\begin{enumerate}
\item $\abs{\mathcal{C}(G,\tau)}=1$ if and only if $\tau$ is precompact. In this case, $(G,\tau)$ is precompact and Mackey.
\item Otherwise $\abs{\mathcal{C}(G,\tau)}\geq 2^\c$. In this case, $\tau$ is the only metrizable topology in $\mathcal{C}(G,\tau)$.
\end{enumerate}

Corollary B1 solves Conjecture 7.6 in \cite{ADM1} for bounded groups.






\vspace{0.2cm}


 Let $G$ be a countable\footnote{better: countably infinite} group. Under Martin Axiom, there exist $2^\c$ precompact topologies of weight $\c$ on $G$. Indeed, there exist $2^{2^{\abs{G}} }=2^\c $ precompact -pairwise different- topologies on $G$. We compute the cardinality of precompact topologies with weight $<\c$. Consider a cardinal\footnote{L: Shouldn't we assume $\kappa$ to be infinite?} $\kappa<\c$. Since every precompact topology of weight $\kappa$ is determined exactly by a subgroup of $G^*$ of cardinality $\kappa$, there exist $\abs{G^*}^\kappa$ precompact topologies of weight $\kappa$. Now, $\abs{G^*}=\c$, hence $\abs{G^*}^\kappa=\c^\kappa=(2^\omega)^\kappa=2^\kappa\stackrel{M.A.}{=}\c$ . Consequently, there exist $2^\c$ precompact topologies of weight $\c$ on $G$.

The following question arises:
\begin{question}
Let $G$ be a countable group. How many of the precompact topologies of weight $\c$ on $G$ are Mackey? How many of them fail to be Mackey?
\end{question}

Although we don't know the answer to this question, we give the following partial answer:

\vspace{0.2cm}
\noindent\textbf{Corollary B 2: } Let $G$ be a bounded countably infinite group. Then, there exist $\c$ precompact topologies of weight $\c$ which are not Mackey.
\vspace{0.2cm}

\begin{proof}
 Since $G$ is bounded and countably infinite, there exists a natural number $L$ such that $G=\bigoplus_{n<\omega} \Z_{m_n}$, where $m_n<L$ for all $n\in\N$. Hence, we can write $G=G_1\times G_2$ where both $G_1$ and $G_2$ are infinite. Since $G_1$ is infinite, there exists a family of $\c$ precompact topologies $\{\tau_i\}$ on $G_1$ whose weight is $<\c$. \footnote{L: Don't we need $=\omega$ in order to obtain that $(G_1,\tau_i)$ is metrizable?}  Consider the discrete topology $d$ in $G_2$ and in $G$ the product topology, namely $\mathcal{T}_i=\tau_i\times d$. The topology $\mathcal{T}_i$ is metrizable (hence, by Theorem B, Mackey) but not precompact (for each $i$). In addition, $\abs{(G,\mathcal{T}_i)^\wedge}=\c$. Hence the topologies $\mathcal{T}_i^+$ are precompact of weight $\c$ and not Mackey.
\end{proof}

\begin{question}
Does there exist a precompact bounded group of weight $\c$ which is Mackey?\footnote{Here again my question:
is $\bigoplus \Z_m$ with the topology induced by  $\prod \Z_m$ Mackey?}
\end{question}

\vspace{0.2cm}

In \S \ref{sectionBounded}, we prove Theorem B and Corollary B1.

\begin{remark}
We shall show in another paper that in the class of linearly topologized groups the Mackey topology exists and that
every metrizable linear topology is the Mackey topology in this class.
\end{remark}

%

For unbounded groups, the picture is completely different. In fact, on any unbounded topological group we can build a metrizable locally quasi-convex group topology which is not Mackey. This is Theorem C, which is the main result in this direction:

\vspace{0.2cm}
\noindent\textbf{Theorem C: }{\em If $G$ is an unbounded group, then $G \not \in \VP$, i.e., there exists a metrizable locally quasi-convex topology on $G$ which is not Mackey.}
\vspace{0.2cm}

The proof of Theorem C, given in \S 3, is based in the following facts:
\begin{itemize}
\item $\Z,\Zpinfty$ and $\Q$ have a metrizable locally quasi-convex non-Mackey topology (see Example \ref{exampleUnboundedDadidi}).
\item Every unbounded group $G$ has a subgroup of the form $\Z,\Zpinfty,\Q$ or $\bigoplus_{n=1}^\infty\Z_{m_n}$,  where $p$ is a prime and
$(m_n)$ is a sequence with $\lim m_n = \infty$ \cite{Fuc}.
\item Open subgroups of Mackey groups are Mackey (Lemma \ref{theorem11Dadidi}).

\item For any sequence $(m_n)$ with $m_n\rightarrow\infty$ the product topology of the group
$\bigoplus_{n=1}^\infty\Z_{m_n}$
is metrizable and non-Mackey (Proposition \ref{theorem12Dadidi}), witnessing $\bigoplus_{n=1}^\infty\Z_{m_n}\not \in \VP$.
\end{itemize}

 As a by-product, the construction in Proposition \ref{theorem12Dadidi} (of a compatible locally quasi-convex  topology finer than the product topology) can be used to provide $\cont$ many pairwise non-isomorphic metrizable
locally precompact linear group topologies on the group $\Q/\Z$ and all its subgroups of infinite rank (compare with Example \ref{exampleUnboundedDadidi}(c), where the single topology with similar properties is not linear).


\subsection{Notation and terminology.}

The symbols $\N,\Z,\Q,\C,\Zpinfty,\T$ will stand for the natural numbers, the integer numbers, the rational numbers, the complex numbers, the Pr\"ufer groups and the unit circle of the complex plane respectively. Using the structure of $\C$, in particular, the real part function $Re(z)$ for $z\in\C$, we let $\T_+=\T\cap\{z\in\C:Re(z)\geq 0\}$ (although we frequently identify $\T$  with the quotient group $\R / \Z$, if  the additive notation is more convenient).
The group $\Z_m$ denotes the cyclic group of $m$ elements. For an abelian group $G$, we denote by $G_d$ the group endowed with its discrete topology
and $G^* = G_d^\wedge = \Hom (G, \T)$.  For a cardinal $\kappa$ we denote by $G^{(\kappa)}$ the direct sum $\bigoplus_\kappa G$ of $\kappa$ copies of $G$.

 For a topological group $(G,\tau)$ its dual group is the group of continuous homomorphisms $\CHom(G,\T)$, and will be denoted by $(G,\tau)^\wedge$ or $G^\wedge$ if the topology is clear. For an abelian group $G$, we say that $\tau$ and $\nu$ are compatible if $(G,\tau)^\wedge=(G,\nu)^\wedge$. For a topological group $(G,\tau)$, we denote by $\mathcal{C}(G,\tau)$ the family of all locally quasi-convex topologies on $G$ which are compatible with $\tau$. A subgroup $H\leq G$ of a topological group is dually closed if for any $x\notin H$, there exists $\phi\in G^\wedge$ satisfying that $\phi(H)=0$ and $\phi(x)\neq 0$.

For an abelian group $G$ and a subgroup $H\leq \Hom(G,\T)$ we shall denote by $\sigma(G,H)$ the weak topology on $G$ induced by the elements of $H$. When the initial group is a topological group, then $G^\wedge\leq \Hom(G,\T)$ is a subgroup and we write $\tau^+$ instead of $\sigma(G,G^\wedge)$. The topology $\tau^+$ is usually called the Bohr topology of $(G,\tau)$.  Obviously, $H$ is dually closed if it is $\tau^+$-closed. If $G$ carries the discrete topology, the corresponding Bohr topology has remarkable properties (see for example \cite{VD1990}).

A group $G$ is said to be bounded if there exists $m\in\N$ satisfying that $mG=0$. In this case, we will call the exponent of $G$ to the minimum positive integer satisfying this condition. For a prime number $p$, we say that a group $G$ is a $p$-group if for every $x\in G$ there exists a natural number $n$ satisfying that $p^nx=0$.

\begin{definition}
Let $(G,\tau)$ be a topological group. We say that $\tau$ is {\em linear} if it has a neighborhood basis consisting in subgroups. In this case, we say that $(G,\tau)$ is {\em linearly topologized}.
\end{definition}


\begin{definition}
Let $(G,\tau)$ be a topological group. We say that a subset $M\subset G$ is {\em quasi-convex} if for every $x\in G\setminus M$, there exists $\phi\in G^\wedge$ satisfying that $\phi(M)\subset\T_+$ and $\phi(x)\notin\T_+$. If  $\tau$ has a neighborhood basis consisting in quasi-convex sets, we say that $\tau$ is
{\em locally quasi-convex}.
\end{definition}

Further, we set for a subset $M\subseteq G\quad M^\triangleright=\{\chi\in G^\wedge:\ \chi(M)\subseteq\T_+\}$ and for
a subset $N\subseteq G^\wedge\quad N^\triangleleft=\{x\in G:\ \chi(x)\in \T_+\ \forall \chi\in N\}$.

For the reader's convenience, we include a proof of the following result, which can be found also in \cite[Proposition 2.1]{AG}.

\begin{proposition}\label{lqc+bdd=linear}\label{proposition7Dadidi}
Let $(G,\tau)$ be a locally quasi-convex bounded group. Then $\tau$ is linear.
\end{proposition}

\begin{proof} Suppose that $mG=0$. 	
Let $U$ be a quasi-convex neighborhood of $0$ in $G$. It is well known that $U^\triangleright$ is an equicontinuous
subset of $G^\wedge$. Hence there exists a neighborhood $W$ of $0$ in $G$ such that $\chi(W)\subseteq ]-\frac{1}{ m},
\frac{1}{ m}[+\Z$ for all $\chi\in U^\triangleright$. Since $m\chi\equiv 0$, we obtain $\chi(W)\subseteq \chi(G)\subseteq
\{\frac{k}{m}+\Z,\ k\in\Z\}$. Combined with the above inclusion, this yields $\chi(W)=\{0\}$ for every $\chi\in U^\triangleright$.
This means $W\subseteq (U^\triangleright)^\bot$. In particular, the subgroup $(U^\triangleright)^\bot\subseteq
(U^\triangleright)^\triangleleft=U$ is open.

%
%
\end{proof}

\begin{corollary}\label{corollary7.1Dadidi}\label{corollary7.2Dadidi}
The notions $\Lin$-Mackey and Mackey are equivalent for bounded groups. 
\end{corollary}

We will consider the following notation for some classes of topological groups:

\begin{itemize}
\item $\Lin $ (resp., $\Lin$-Mackey) -- for the class of linearly topologized Hausdorff groups.
\item $\Lp $  (resp., $\Lp$-Mackey) -- for the class of locally precompact groups.
\item $\LQC $  (resp., $\LQC$-Mackey) -- for the class of locally quasi convex Hausdorff groups.
\end{itemize}

Since locally quasi-convex groups generalize locally convex spaces, the most natural class $\G$ to study Mackey topologies is the class of locally quasi-convex groups. Hence, we write simply Mackey for $\LQC$-Mackey.

%
%
%
%
%
\section{The bounded case: proofs of Theorem B and Corollary B1}\label{sectionBounded}

\begin{proposition}\label{open/closed}\label{proposition1Dadidi}
Let $G$ be an abelian group. If $\tau , \tau'$ are compatible topologies on $G$, they share the same dually closed subgroups. 
\end{proposition}

\begin{proof}
The assertion trivially follows from these two facts:
\begin{itemize}
\item[(a)] both topologies have the same weak topology $\sigma(G,G^\wedge)$;
\item[(b)] a subgroup $H$ of $G$ is dually closed (with respect to any of these two topologies) precisely when it is $\sigma(G,G^\wedge)$-closed.
\end{itemize}
\end{proof}

\bigskip

Next, we want to show that a metrizable locally quasi-convex  bounded group is Mackey. For the proof we need the following Proposition which is interesting on its own:

\begin{proposition}
Let $(G,\tau)$ be a torsion, metrizable non-discrete abelian group. Then there exists a homomorphism $f:G\to\T$ which is not continuous.
\end{proposition}

\begin{proof} Since $G$ is metrizable and $\tau$ is not the discrete topology, we can choose a null-sequence $(a_n)_{n\in\N}$ which satisfies $a_n\not=a_m$ whenever $n\not=m$ and $a_n\not=0$ for all $n\in\N$. Since every finitely generated torsion group is finite we may additionally assume that $a_{n+1}$ does not belong to the subgroup generated by $a_1,\ldots,a_n$ for every $n\in\N$.

We are constructing inductively a homomorphism $f_n:\langle a_1,\ldots,a_n\rangle\to \T$ in the following way:

Let $f_1:\langle a_1\rangle\to\T$ be a homomorphism which satisfies $f_1(a_1)\notin\T_+$. This is possible, since $a_1\not=0$.

Suppose that $f_n:\langle a_1,\ldots,a_n\rangle\to\T$ is a homomorphism, which satisfies $f_n(a_k)\notin\T_+$ for all $1\le k\le n$.

\noindent If $\langle a_{n+1}\rangle \cap \langle a_1,\ldots,a_n\rangle=\{0\}$, we can define a homomorphism
$f_{n+1}:\langle a_1,\ldots,a_{n+1}\rangle\to\T$ which extends $f_n$ and satisfies $f_{n+1}(a_{n+1})\notin\T_+$.

\noindent If $\langle a_{n+1}\rangle \cap \langle a_1,\ldots,a_n\rangle\not=\{0\}$, we choose the minimal natural
number $k\in\N$ such that $ka_{n+1}=k_1a_1+\ldots+k_na_n$ for suitable $k_1,\ldots,k_n\in\Z$. By assumption, $2\le k< {\rm ord}(a_{n+1})$. We define $f_{n+1}(a_{n+1})$ such that $f_{n+1}(a_{n+1})\notin\T_+$ and
$f_{n+1}(ka_{n+1})=f_n(k_1a_1+\ldots+k_na_n)$ and obtain a homomorphism $f_{n+1}:\langle a_1,\ldots,a_{n+1}\rangle\to\T$ which extends $f_n$.

In this way, we obtain a homomorphism $\wt{f}:\langle  \{a_n:\ n\in\N\}\rangle \to\T$ which can be extended to a homomorphism $f:G\to \T$. By construction,  $f$ satisfies:
$$
\forall n\in\N\  f(a_n)\notin\T_+.
$$
So $f$ is not continuous.
\end{proof}

\begin{remark}
(a) Using that no non-trivial sequence can converge in the Bohr topology $\mathcal P_G$ of the discrete group of $G$
(\cite{Flor}) one can obtain a proof of a stronger version of the above proposition
with ``torsion" omitted. (Just take any non-trivial null sequence $(a_n)$ as in the above proof and note
that it cannot converge in  $\mathcal P_G$, hence there exists a character $\chi$ of $G$
witnessing that, i.e., such that $\chi(a_n) \not \to 0$  in $\T$; then
$\chi$ cannot be $\tau$-continuous, as $a_n \to 0$ in $(G,\tau)$.)

(b) If we reinforce the hypothesis of the above proposition replacing ``torsion" by ``bounded", the proof can
be simplified in the following way. Passing to a subsequence of $(a_n)$ one can achieve to have
$\langle a_{n+1}\rangle \cap \langle a_1,\ldots,a_n\rangle=\{0\}$ to avoid the more complicated second case ($\langle a_{n+1}\rangle \cap \langle a_1,\ldots,a_n\rangle\not =\{0\}$) in the above proof.
\end{remark}

\noindent\textbf{Proof of Theorem B: }
Let $\tau$ be a locally quasi-convex topology on a bounded group $G$. We have to prove that $(G,\tau)$ is a Mackey group provided that: (i) $\tau$ is metrizable; or (ii) $\abs{(G,\tau)^\wedge}<\c$.

(i) Suppose that $\tau$ is metrizable. Let $\tau'$ be another locally quasi-convex group topology on $G$ compatible with $\tau$. According to \ref{lqc+bdd=linear}, $\tau$ and $\tau'$ are linear topologies. We have to show that $\tau'\subseteq \tau$.
So fix a $\tau'$-open subgroup $H$ of $G$ and denote by $q:G\to G/H$ the canonical projection. As $H$ is open in $(G,\tau')$, it is also dually closed  in $(G,\tau')$ and according to \ref{open/closed} also in $(G,\tau)$. Let
$f:G/H\to\T$ be an arbitrary homomorphism. Then $f\circ q\in (G,\tau')^\wedge=(G,\tau)^\wedge$. Let $\hat{\tau}$ denote the
Hausdorff quotient topology on $G/H$ induced by $\tau$. Of course, $(G/H,\hat{\tau})$ is a metrizable bounded group. Since every homomorphism $f:(G/H,\hat{\tau})\to \T$ is continuous, we obtain that $G/H$ is discrete by the above Proposition. This implies that $H$ is $\tau$-open. Hence, $\tau'\subseteq \tau$.

\vspace{0.1cm}

(ii) Now, suppose that $\abs{(G,\tau)^\wedge}<\c$. By Proposition \ref{lqc+bdd=linear}, $\tau$ is linear.
 Suppose that $\tau$ is not precompact. Then there exists an open subgroup $U\leq G$ satisfying that $\abs{G/U}\geq\omega$. Since $U$ is open, the quotient $G/U$ is discrete. Hence $(G/U)^\wedge$ is compact and infinite, therefore $\abs{(G/U)^\wedge}\geq\cont$. Since $\abs{G^\wedge}\geq\abs{(G/U)^\wedge}$ the inequality $\abs{(G,\tau)^\wedge}\geq\cont$ follows. This contradiction shows that $\tau$ is precompact.

 Let $\tau'$ be another locally quasi-convex group topology on $G$, which is compatible with $\tau$. According to
  \ref{lqc+bdd=linear}, $\tau'$ is linear.
  As $(G,\tau')^\wedge=(G,\tau)^\wedge$ is countable, the above argument applied to $\tau'$
  shows that $\tau'$ is precompact as well.
  Since $\tau$  and $\tau'$ are compatible, they must coincide. Therefore, $\tau$ is Mackey.


\bigskip

\noindent\textbf{Proof of Corollary B1:}
We have to prove that if $(G,\tau)$ is a metrizable, locally quasi-convex and bounded group, then:
\begin{enumerate}
\item $\abs{\mathcal{C}(G,\tau)}=1$ if and only if $\tau$ is precompact.
\item otherwise $\abs{\mathcal{C}(G,\tau)}\geq 2^\c$; in this case, $\tau$ is the only metrizable topology in $\mathcal{C}(G,\tau)$.
\end{enumerate}

If $\abs{\mathcal{C}(G,\tau)}=1$ it is clear that $\tau$ is precompact (for any dual pair $(G,G^\wedge)$ there exists one precompact topology). Conversely, if $\tau$ is precompact, Theorem B implies that $\tau$ is Mackey. Hence it is clear that $\abs{\mathcal{C}(G,\tau)}=1$.

The proof of 2. relies on the fact that $\tau$ is linear (Proposition \ref{lqc+bdd=linear}). Assume that $(G,\tau)$ is not precompact. Then, there exists a $\tau$-open subgroup $U$ such that $G/U$ is infinite. Since $U$ is open, $G/U$ is discrete. By Theorem 3.10 in \cite{ADM1}, the family $\mathcal{C}((G/U)_d)$ can be embedded in $ \mathcal{C}(G,\tau)$. We denote by $\mathcal{F}il_{G/U}$ the set of all filters on $G/U$ and we recall that $\abs{\mathcal{F}il_{G/U}}=2^{2^{\abs{G/U}}}$. By Theorem 4.5 in \cite{ADM1}, we have that $\abs{\mathcal{C}((G/U)_d)}=\abs{\mathcal{F}il_{G/U}}$, since $G/U$ has infinite rank (being an infinite bounded group). So $$\abs{\mathcal{C}(G,\tau)}\geq\abs{\mathcal{C}((G/U)_d)}=2^{2^{\abs{G/U}}}\geq 2^{2^\omega}=2^\c$$
holds. It is also clear from Theorem B, that $\tau$ is the only metrizable topology in $\mathcal{C}(G,\tau)$.

\section{The unbounded case: proof of Theorem C}\label{sectionUnbounded}

In order to prove that unbounded groups do not fit into Varopoulos paradigm, we need first to prove that the Mackey topology is "hereditary" for open subgroups. In other words, an open subgroup of a Mackey group is also a Mackey group in the induced topology. To this end we recall next an standard   extension of a  topology from a subgroup to the whole group.
\begin{definition}
Let $G$ be an abelian group and $H$ a subgroup. Consider on $H$ a group topology $\tau$. We define the extension $\overline{\tau}$ of $\tau$ as the topology on $G$ which  has as a basis  of neighborhoods  at zero the neighborhoods of  $(H, \tau)$.
\end{definition}

If the topology $\tau$ on $H$ is metrizable, the topology $\overline{\tau}$ on $G$ is metrizable. Analogously, if $\tau$ is locally quasi-convex, then $\overline{\tau}$ is also locally quasi-convex (indeed, a subset which is quasi-convex in $H$ is quasi-convex in $G$, due to the fact that $H$ is an open subgroup in $\overline{\tau}$).

\begin{lemma}\label{openSubgroupsAreMackey}\label{theorem11Dadidi}\label{corolarioSubgrupoAbiertoMackey}\label{lemma11.1Dadidi}
Let $H$ be an open subgroup of  the topological group $(G,\tau)$. If $\tau$ is the Mackey topology, then the induced topology $\tau_{\mid H}$ is also Mackey. As a consequence, if $\nu$ is a metrizable, locally quasi-convex {\em  non-Mackey} group topology on $H$, then there exists a metrizable locally quasi-convex-group topology on $G$ which is not Mackey, namely $\overline{\nu}$.
\end{lemma}
\begin{proof}
In order to prove that $(H, \tau_{|H})$ is also a Mackey group take a locally quasi-convex topology $\nu$ in $H$ such that  $(H, \nu)^\wedge = (H,\tau_{\mid H})^\wedge$. Let $\overline{\nu}$ be the extension of $\nu$ to $G$.

 Let us prove that $(G,\overline{\nu})^\wedge = (G, \tau)^\wedge$. Take $\phi \in (G,\overline{\nu})^\wedge $. Then $\phi_{|H}$ is $\nu$-continuous, and, since $\nu$ and $\tau_{\mid H}$ are compatible, $\phi_{\mid H}$ is also a character of $(H, \tau_{|H})$. Since $H$ is $\tau$-open, $\phi$ is  $\tau$-continuous. Thus $(G,\overline{\nu})^\wedge \leq (G, \tau)^\wedge$.  The same argument can be used to prove  the converse inequality, taking into account that  $H$ is  an open subgroup in $\overline{\nu}$. So $\overline{\nu} $ and $\tau$ are compatible topologies.

Since $(G, \tau)$ is Mackey, we obtain that $\overline{\nu}\leq \tau$ and this inequality also holds for  their restrictions to $H$,  $\nu \leq \tau_{|H}$. Hence $(H, \tau_{|H})$ is Mackey.

Let $\nu$ be a metrizable locally quasi-convex group topology on $H$ which is not Mackey and $\overline{\nu}$ its extension to $G$. Clearly, $H$ is $\overline{\nu}$-open,  and $\overline{\nu}$  is a  metrizable  locally quasi-convex topology on $G$. If $\overline{\nu}$ were Mackey, by our previous arguments, $\nu$ should be Mackey as well, a contradiction.
\end{proof}

We do not know if the converse of Lemma \ref{openSubgroupsAreMackey} holds in general:

\begin{question}
Let $(G,\tau)$ be a topological group and $H\leq G$ an open subgroup. If $\tau_{\mid H}$ is the Mackey topology, is $\tau$ the Mackey topology for $G$?
\end{question}

We prove it for the class of bounded groups.

%
%

\begin{proposition}\label{Daniel}
Let $(G,\tau)$ be a bounded topological abelian group and $H\leq G$ an open subgroup. If $(H,\tau_{\mid H})$ is Mackey, then $(G,\tau)$ is Mackey.
\end{proposition}
\begin{proof}
Let $\tau'$ be a locally quasi-convex group topology on $G$ which is compatible with $\tau$.
 Note first that $\tau$, $\tau'$, $\tau_{\mid H}$ and $\tau'_{\mid H}$ are linear topologies (Proposition \ref{lqc+bdd=linear}). We are going to prove that $\tau_{\mid H}$ and $\tau'_{\mid H}$ are compatible.

 Let $\chi:(H,\tau)\to\T$ be  a continuous character. Since $H$ is an open subgroup of $(G,\tau)$, there exists a continuous character $\wt{\chi}\in (G,\tau)^\wedge$ extending $\chi$. Since $\tau'$ is compatible with $\tau$, we obtain
 $\wt{\chi}\in(G,\tau')^\wedge$ and hence $\chi|_H\in(H,\tau'|_H)^\wedge$.

 Assume conversely, that $\chi\in(H,\tau'|_H)^\wedge$.  Since $\chi$ is $\tau'$-continuous, there exists an open subgroup $U\in\tau'$ such that $\chi(H\cap U)=\{0\}$. Hence the mapping   $\chi':H+U\to\T,\ h+u\mapsto
 \chi(h)$ ($h\in H$ and $u\in U$) is a well defined homomorphism which extends $\chi$. As $\chi'(U)=\{0\}$, we obtain that $\chi$ is $\tau'|_{H+U}$ is continuous. Since $H+U$ is an open subgroup of $(G,\tau')$, $\chi'$ can be extended to a continuous homomorphism $\wt{\chi}\in(G,\tau')^\wedge=(G,\tau)^\wedge$.
 As above, this shows that $\chi|_H\in (H,\tau|_H)^\wedge$.

 Since $\tau_{\mid H}$ is Mackey, we have that $\tau'_{\mid H}\leq\tau_{\mid H}$. Since  $H$ open in $\tau$, it follows that $\tau'\leq\ol{\tau'}\le \ol{\tau}=\tau$ where $\ol{\tau}$ and $\ol{\tau'}$ denote the extensions to $G$. Since $\tau_{\mid H}$ locally quasi-convex and $H$ an $\tau$-open subgroup, the topology  $\tau$ is locally quasi-convex and consequently the Mackey topology on $G$.
\end{proof}

\begin{proposition}\label{theorem12Dadidi}
Let $(m_n)$ be a sequence of natural numbers with $m_n\rightarrow\infty$ and $G = \bigoplus\Z_{m_n}$. Then
the metrizable locally quasi-convex topology $\tau$ on $G$ inherited from the product $\prod_\omega\Z_{m_n}$ is not Mackey.
\end{proposition}

\begin{proof}
  The idea of the proof is to fix appropriately a   sequence $\mathbf{c}\subset G^\wedge = \bigoplus_\omega\Z_{m_n}^\wedge$. If we denote its range again by $\mathbf{c}$ and define $\tauc$  the topology of uniform convergence on $\mathbf{c}$, $\tauc$ is non-precompact, and strictly finer than
  $\tau$, but still compatible. Hence, $\tau$ is not Mackey.


 Observe that $G^*=\prod\Z_{m_n}^\wedge$ and $G^\wedge=\bigoplus_\omega\Z_{m_n}^\wedge$.
 Clearly, $\Z_{m_n}^\wedge\cong \Z_{m_n}$,
  and for convenience we make the following identifications:
$$\Z_{m_n}^\wedge=\left(-\frac{m_n}{2},\frac{m_n}{2}\right]\cap\Z$$
 and
\begin{equation}\label{dagdag}
\Z_{m_n}=\left\{\frac{k}{m_n}: k\in\left(-\frac{m_n}{2},\frac{m_n}{2}\right]\cap\Z\right\}.
\end{equation}

We can assume without loss of generality that the sequence $(m_n)$ is non-decreasing, since a rearrangement of the sequence $(m_n)$ leads to an isomorphic group.

For $x=(x_1,x_2,\dots)\in G$ and $\chi=(\chi_1,\chi_2,\dots)\in G^\wedge$, we have $\chi(x)=\sum_{n=1}^\infty\chi_n x_n+\Z$.

For $n\in\N$, let $(e_n)=(0,\dots,0,
{1},0,\dots)\in G^\wedge$, with $1$ in $n$-th position and define
$\mathbf{c}=\{\pm e_n:n\in\N\}\cup\{0\}\subset G^\wedge$.\footnote{Here it was stated that ${\bf c}$ is a compact set without any reference to the topology. In $G^\wedge$ it is surely not compact, as the character group is
discrete. So I eliminated "compact".}
Identifying $\T$ with $(-\frac{1}{2},\frac{1}{2}]$, we let $\T_m= \{x\in \T: \abs{x}\leq\frac{1}{4m}\}$ and
\begin{equation}\label{(99)}
V_m
:=\{x\in G:e_n(x)\in\T_m\mbox{ for all } n\in\N\}\stackrel{e_n(x)=x_n}{=}\left\{x\in G:\abs{x_n}\leq\frac{1}{4m}\mbox{ for all }n\in\N\right\}.\quad\quad
\end{equation}

The family $\{V_m:m\in\N\}$ is a neighborhood basis of $0$ for the topology $\tauc$ of uniform convergence on $\mathbf{c}$.

Writing the elements $x=(x_n)$ of $G$, with $x_n=\frac{k_n}{m_n}$ (as in (\ref{dagdag})), one has $V_m=\left\{x\in G:\abs{k_n}\leq\frac{m_n}{4m}\mbox{ for all } n\right\}$.

To conclude the proof of the proposition we need the following lemma:

\begin{lemma}\label{lemma12.1Dadidi}
$(G,\tauc)^\wedge=\bigoplus_\omega\Z_{m_n}^\wedge$.
\end{lemma}

\begin{proof}
The fact $\bigoplus_\omega\Z_{m_n}^\wedge\leq (G,\tauc)^\wedge$ is straightforward. Indeed, let $\chi\in\bigoplus_\omega\Z_{m_n}^\wedge$. Then there exists $n_1\in\N$ such that $\chi_n=0$ if $n\geq n_1$. Let $m=m_{n_1}$, hence $x_1=\dots=x_{n_1-1}=0$ for all $x\in V_m$. Then $\chi(V_m)=\{0_\T\}$.

Let $\chi=(\chi_n)\in\prod_\omega\Z_{m_n}^\wedge\setminus\bigoplus_\omega\Z_{m_n}^\wedge$. Then the set
$A:=\{n\in\N:\chi_n\neq 0\}$ is infinite. Our aim is to prove that $\chi\not\in(G,\tauc)^\wedge$. We argue by
contradiction,  assuming that $\chi\in(G,\tauc)^\wedge$. Then $\chi\in V_m\tr:=\{\phi\in (G,\tauc)^\wedge:\phi(V_m)\subset\T_+\}$ for some  $m\in\N$. Hence, we must
find $x\in V_m$ with $\chi(x)\notin\T_+$.

To this end, we pick
\begin{equation}\label{n03Abril}
n_0:= \min\left\{n\in\N:\frac{m_n}{4m}\geq\frac{5}{2}\right\} .
\end{equation} Equivalently, $\dis \frac{1}{m_n}\leq\frac{1}{10m}$ if $n\geq n_0$.

We describe next in $(a)$ some elements of $V_m$, which will be suitable for our purposes and in $(b)$ a condition   for the components of $\chi$.

\begin{claim}\label{LastClaim}
\begin{itemize}
\item[(a)] If $x_n=0$ for $n<n_0$ and $x_n\in\{0,\pm\frac{1}{m_n}\}$ for all $n\ge n_0$, then $x=(x_n)\in V_m$.
\item[(b)] $\dis \frac{\abs{\chi_n}}{m_n}\leq\frac{1}{4}\mbox{ for all }n_0\leq n.$
\end{itemize}
\end{claim}

\begin{proof}
(a)
 Indeed, since $\dis \frac{1}{m_n}\leq\frac{1}{m_{n_0}}\leq\frac{1}{10m}$ if $n\geq n_0$, it follows  from (\ref{(99)}) that $x\in V_m$.

(b) By item (a), we have that $x=(0,\dots,\stackrel{(n)}{\frac{1}{m_n}},\dots)\in V_m$. Since $\chi\in V_m\tr$, it follows that $\dis \abs{\chi(x)}=\frac{\abs{\chi_n}}{m_n}\leq\frac{1}{4}$.
\end{proof}

To continue the proof of Lemma \ref{lemma12.1Dadidi}, define, for each $k\in\left\{0,1,\dots,\frac{m_{n_0}}{2}\right\}$, the set

\begin{equation}\label{Ak3Abril}
A_k:=\left\{n\in A:\frac{k}{m_{n_0}}<\frac{\abs{\chi_n}}{m_n}\leq\frac{k+1}{m_{n_0}}\right\}.
\end{equation}

Then $\left\{A_k: 0 \leq k\leq \frac{m_{n_0}}{2} \right\}$ is a finite partition of $A$ and at least one $A_k$ is infinite,
as $A$ is infinite. In order to define an element $x\in V_m$ which satisfies  $\chi(x)\notin\T_+$ we distinguish
the cases $k\ge 1$ and $k=0$.

\vspace{0.2cm}

\underline{Case 1: $\mathbf {k\geq 1.}$} Since $A_k$ is infinite, one can find $n_0\le n_1 < n_2 < \dots <n_{m_{n_0}}$ in $A_k$.
 Define
\[ x_{n_i} :=\frac{1}{m_{n_i}}\sign(\chi_{n_i})\]
By the definition of $A_k$, we have $\dis \frac{k}{m_{n_0}}<\chi_{n_i}(x_{n_i})$  for $\dis 1\leq i\leq m_{n_0}$.
Thus,
 $$
 \sum_{i=1}^{m_{n_0}}\chi_{n_i}(x_{n_i})> k \geq 1.
 $$
 Taking into account that each summand in the above sum satisfies $\chi_n(x_n)\leq\frac{1}{4}$ by Claim \ref{LastClaim}(b), we conclude that there exists
a natural $N\leq m_{n_0}$ such that
$$
\frac{1}{4}<\sum_{i=1}^N\chi_{n_i}(x_{n_i})<\frac{3}{4}.
$$
Define
\[ x'_n:= \left\{ \begin{array}{lcl}
         \frac{1}{m_{n_i}}\sign(\chi_{n_i}) &\ \mbox{ when \ $n=n_i$ \   for some }\ & 1\leq i\leq N\\
        0 &   \mbox{otherwise} &
         \end{array} \right. \]

By Claim \ref{LastClaim}(a), $x'=(x'_n)\in V_m$ and by construction $\chi(x')\notin\T_+$.

\vspace{0.2cm}

\underline{Case 2: $\mathbf{k=0.}$}
\vspace{0.2cm}

That is: $A_0$ is infinite. Choose $J\subset A_0$ with $\abs{J}=2m$. For $n\in J$, define
$$
j_n:=\frac{\dis \left\lfloor\frac{m_n}{4m\abs{\chi_n}}\right\rfloor \sign(\chi_n)}{m_n}.
$$
First, note that for all $n \in J$ one has $j_n\neq 0$, as $\dis \frac{m_n}{4m\abs{\chi_n}}
{\geq}\frac{m_{n_0}}{4m}
>2$, where the last inequality is due to (\ref{n03Abril}).
On the other hand,  \begin{equation}\label{Estimacion13Abril}\abs{j_n}\leq\abs{\frac{\dis \frac{m_n}{4m\abs{\chi_n}}}{m_n}}=\frac{1}{4m\abs{\chi_n}}.\end{equation}
This gives $\abs{j_n}\leq\frac{1}{4m}$, as  $\abs{\chi_n}\geq 1$. Define
\[ x_n:= \left\{ \begin{array}{lcl} j_n &\ \mbox{ when \ $n\in J$}\\
        0 &   \mbox{otherwise} &
         \end{array} \right. \]

Then, $x = (x_n)\in V_m$, by equation (\ref{(99)}).

In addition, \begin{equation}\label{Estimacion23Abril}\abs{x_n}\geq\frac{\dis \frac{m_n}{4m\abs{\chi_n}}-1}{m_n}=\frac{1}{4m\abs{\chi_n}}-\frac{1}{m_n}.\end{equation}

Combining (\ref{Estimacion13Abril}) and (\ref{Estimacion23Abril}), we get, for $n\in J$:$$\frac{3}{20m}=\frac{1}{4m}-\frac{1}{10m}\leq\frac{1}{4m}-\frac{1}{m_{n_0}}\stackrel{n\in A_0}{\leq}\frac{1}{4m}-\frac{\abs{\chi_n}}{m_n}\stackrel{(\ref{Estimacion23Abril})}{\leq}
|\chi_n|(|x_n|)=
\chi_n(x_n)\stackrel{(\ref{Estimacion13Abril})}{\leq}\frac{1}{4m}.$$

Applying these  inequalities to $\chi(x)=\dis \sum_{n\in J}\chi_nx_n$ gives
$$
\frac{1}{4}<\frac{3}{10}=\frac{3}{20m}\abs{J}\leq\sum_{n\in J}\chi_nx_n=\chi(x)\leq\frac{1}{4m}\abs{J}=\frac{1}{2}.
$$
This implies that $\chi(x)\notin\T_+$, a contradiction in view of $x\in V_m$.

In both cases this proves that  $\chi\notin(G,\tauc)^\wedge$.
\end{proof}
Hence $\dis(G,\tauc)^\wedge=\bigoplus_\omega \Z_{m_n}^\wedge$.
\end{proof}

\noindent\textbf{Proof of Theorem C: } We have to prove that for any unbounded abelian group $G$, there exists a metrizable locally quasi-convex group topology which is not Mackey. We divide the proof in three cases, namely:
\begin{itemize}
\item[$(a)$] $G$ is non reduced.
\item[$(b)$] $G$ is non torsion.
\item[$(c)$] $G$ is reduced and torsion.
\end{itemize}

$(a)$ If $G$ is non reduced, write $d(G)$ its maximal divisible subgroup. Then $d(G)\cong\bigoplus_{\alpha}\Q\oplus\bigoplus_{i=1}^\infty\Z(p_i^\infty)^{(\beta_i)}$ \cite[Theorem 21.1]{Fuc}.

If $\alpha\neq 0$, then we there exists $H\leq G$, where $H\cong\Q$. Since, there exists a metrizable non-Mackey topology on $\Q$ (\cite{DlB}), applying Lemma \ref{lemma11.1Dadidi}, $G$ is not Mackey as well.

If $\alpha=0$, then there exists $p_i$ satisfying $\Z(p_i^\infty)\leq d(G)\leq G$. Since there exists a metrizable, non-Mackey topology on $\Z(p_i^\infty)$ (\cite{DlB}), there exists (by Lemma \ref{lemma11.1Dadidi}) a metrizable topology on $G$ which is not Mackey.

$(b)$ If $G$ is non torsion, there exists $x\in G$, such that $\hull x\cong\Z$. Equip $\hull\Z$ with the $\b$-adic topology, which is metrizable but not Mackey (\cite{AD}). 
Applying Lemma \ref{lemma11.1Dadidi}, we get the result.

$(c)$ We can write $G$ as the direct sum of its primary components, that is $G=\bigoplus_pG_p$. Consider two cases:

$(c.1)$ $G_p\neq 0$ for infinitely many prime numbers $p$. Then, for some infinite set $\Pi$ of prime numbers,we have $\bigoplus_{p\in\Pi}\Z_p\hookrightarrow G$.  By Lemma \ref{theorem12Dadidi}, there exists a metrizable locally quasi-convex group topology on $\bigoplus_{p\in\Pi}\Z_p$ which is not Mackey. Hence, by  Lemma \ref{lemma11.1Dadidi}, there exists a metrizable locally quasi-convex group topology which is not Mackey on $G$.

$(c.2)$ $G_p\neq 0$ for finitely many $p$. Then, one of them is unbounded, say $G_p$. Let $B:=\bigoplus_{n=1}^\infty\Z_{p^n}^{(\alpha_n)}$ be a basic subgroup of $G_p$, i.e., $B$ is pure and $G_p/B$ is divisible \cite{Fuc}. Let us recall the fact that purity of $B$ means that $p^nB=p^nG\cap B$ for every $n\in \N$.

 We aim to prove that $B$ is unbounded. Assume for contradiction that $B$ is bounded, so $p^nB=0$ for some $n\in \N$. Then $\{0\}=  p^nB=p^nG\cap B$. One  has $(p^nG_p+B)/B\cong p^n G_p/(p^n G_p\cap B)= p^n G_p/p^n B\cong p^n(G_p/B)=G_p/B$; the last congruence holds,  since $B$ is pure in $G_p$ and the last equality is a consequence of the divisibility of $G_p/B$.\footnote{L: I added some intermediate steps and wrote $\cong$ instead of $=$;
 the old version is:\\
 By divisibility of $G_p/B$, one has $(p^nG_p+B)/B= p^n(G_p/B)=G_p/B$} This means that $p^nG_p+B=G_p$. Since $p^nG_p\cap B=0$, this is equivalent to $G_p=p^nG_p\oplus B$.
Then $p^nG_p\cong G_p/B$ is a divisible subgroup of the reduced group $G_p$. Hence $p^nG_p\cong G_p/B=0$ and $G_p = B$ is bounded, a contradiction.

 Since $B$ is unbounded there exists a sequence  $(n_i)\subseteq\N$ with $\alpha_{n_i}\neq 0$ such that $\bigoplus_{i=1}^\infty\Z_{p^{n_i}}\leq B\leq G$. By Lemma \ref{theorem12Dadidi}, there exists a metrizable locally quasi-convex group topology on $\bigoplus_{i=1}^\infty\Z_{p^{n_i}}$ which is not Mackey. Hence, by  Lemma \ref{lemma11.1Dadidi}, there exists a metrizable locally quasi-convex group topology which is not Mackey on $G$.

\vspace{0.2cm}

Now we can prove the main theorem.

\bigskip

\noindent \textbf{Proof of Theorem A:}
%
We have to prove that for an abelian group  $G$ the following assertions are equivalent:

\begin{itemize}
\item[$(i)$] Every metrizable locally quasi-convex group topology on $G$ is Mackey.

\item[$(ii)$] $G$ is bounded.
\end{itemize}

$(i)$ implies $(ii)$ is Theorem C.
$(ii)$ implies $(i)$ is Theorem B.

\end{document}